\documentclass[11pt]{amsart}
\usepackage{amsfonts, amsmath, amssymb}

\theoremstyle{plain}
\newtheorem{thm}{Theorem}[section]
\newtheorem{lem}[thm]{Lemma}

\theoremstyle{remark}
\newtheorem{rem}[thm]{Remark}

\theoremstyle{definition}

\newtheorem{nota}[thm]{Notation}
\newtheorem{dfn}[thm]{Definition}
\newtheorem{conj}[thm]{Conjecture}
\newtheorem{hypo}[thm]{Assumption}

\newcommand{\bbR}{\mathbb R}

\newcommand{\bi}{\mathbf i}

\newcommand{\br}{\mathbf r}
\newcommand{\bs}{\mathbf s}
\newcommand{\bt}{\mathbf t}
\newcommand{\bx}{\mathbf x}

\newcommand{\cG}{\mathcal G}

\newcommand{\rA}{\mathrm A}
\newcommand{\rE}{\mathrm E}

\newcommand{\rG}{\mathrm G}

\newcommand{\rM}{\mathrm M}

\newcommand{\Supp}{\mathrm {Supp} \:}
\newcommand{\Span}{\mathrm {Span} \:}

\DeclareMathOperator{\id}{i d}
\newcommand{\Forall}{\forall \:}

\begin{document}
\title[Regularity of heat kernels]{Exponential coordinates and regularity of groupoid heat kernels}

\author[B.K. So]{Bing Kwan SO}

\begin{abstract}
We prove that on an asymptotically Euclidean boundary groupoid, 
the heat kernel of the Laplacian is a smooth groupoid pseudo-differential operator.
\end{abstract}

\maketitle

\section{Introduction}
In this notes we study the regularity of certain groupoid heat kernels.

\subsection{The problem}
We begin with stating our problem, recalling some basic notations along the way. 
Our motivation will be explained in the next section. 

First of all, we briefly recall the notion of geoupoid pseudo-differential operators and heat kernels.
The theory is developed by Nistor, Weinstein and Xu \cite{NWX;GroupoidPdO} and is considered classical.

Let $\cG \rightrightarrows \rM $ be a Lie groupoid. We shall assume $\cG$ is Hausdorff.
Denote the source and the target maps by $\bs $ and $\bt $ respectively.
We shall use the convention $\bs (a b) = \bs (b), \forall a, b \in \cG, \bt (b) = \bs (a)$.
For each $x \in \rM$, write $\cG _x := \bs ^{-1} (x) $ for the $\bs$-fiber over $x$,
and $T ^\bot \cG := \{ X \in T \cG : d \bs (X ) = 0 \}$.
Note that $T ^\bot \cG $ is just the tangent space of all the $\cG _x $.

\begin{dfn}
A pseudo-differential operator $\varPsi $ on a groupoid $\cG$ of order $\leq m$ 
is a smooth family of pseudo-differential operators $\{ \varPsi _x \}_{x \in \rM}$,
where $\varPsi _x \in \Psi ^m ( \cG _{x} )$,
and satisfies the right invariance property
$$ \varPsi _{\bs (a)} (\br _a^* f) = \br _g^* \varPsi _{\bt (a)} (f), 
\quad \forall a \in \cG, f \in C^\infty_c (\cG _{\bs (a)}).$$
If, in addition, all $\varPsi _x $ are classical of order $m$, then we say that $\varPsi $ is classical of order $m$.
\end{dfn}

\begin{dfn}
The operator $\varPsi $ is called uniformly supported if the set
$$ \{ a b^{-1} : (a, b) \in \Supp (\varPsi) \} $$
is a compact subset of $\cG$.
\end{dfn}

One particularly interesting groupoid differential operator is the Laplacian.
Let $\rA \to \rM$ be the Lie algebroid of $\cG$ with anchor map $\nu : \rA \to T \rM $,
and $\rE $ be a vector bundle over $\rM$.
Let $g _\rA $ be a Riemannian metric on $\rA $ and $\nabla ^\rE $ be an $\rA $-connection on $\rE $
(for the definition of an $\rA$-connection and other details, see \cite{Fern'd;HoloAndChar}). 
By right invariance, one obtains a (family of) Riemannian metric and connection, 
which we still denote by $\cG _\rA $ and $\nabla ^\rE $ respectively, on $\cG _x $.
\begin{dfn}
\label{LapDfn}
The {\it Laplacian} $\Delta ^\rE $ is the family of operators $\{ \Delta ^\rE _x \}_{x \in \rM}, $
where
$$ \Delta ^\rE _x := \sum _{i=1}^n (\nabla ^\rE _{X_i } \nabla ^\rE _{X_i } - \nabla ^\rE _{\nabla ^\rE _{X_i} X_i}), $$
and $X_i$ is any local orthonormal basis of $T \cG _x $.
\end{dfn}
Note that $\Delta ^\rE $ is elliptic, and its principal symbol does not depend on the chosen connection $\nabla ^\rE$.

In the shortest term, the heat kernel of $\Delta ^\rE $ is defined to be
\begin{dfn}
The family $Q := \{ Q _x \} $, where for each $x \in \rM $, $Q _x $ is the heat kernel of $\Delta _x ^\rE $,
is called the {\it heat kernel} of $\Delta ^\rE $. 
\end{dfn}

Equivalently, one can describe the operator $Q$ as a reduced kernel.
\begin{dfn}
\label{RedKer}
For any $\varPsi = \{ \varPsi _x \}_{x \in \rM} \in \Psi ^\infty (\cG)$.
The {\it reduced kernel} of $\varPsi $ is defined to be the distribution
$$ K_\varPsi (f) :=  \int _\rM \mathbf u^* (\varPsi (\mathbf i^* f)) (x) \: \mu_\rM (x), 
\quad f \in C^\infty _c (\cG),$$
where $\mathbf i$ and $\mathbf u$ denote respectively the inversion and unit inclusion.
\end{dfn}

The important point here to note is that,
while each $Q _x $ is defined by a smooth kernel (not compactly supported),
$Q _x $ may not be a smooth family. 
In terms of reduced kernel, it is not clear that the reduced kernel of $Q$ is smooth,
although it is easy to see that $Q $ is continuous.

Therefore, the regularity problem amounts to proving
\begin{conj}
\label{MainProblem}
Given certain groupoid $\cG \rightrightarrows \rM $ with $\rM$ compact, 
the reduced heat kernel of $Q$ is smooth. 
\end{conj}

The main difficulty of establishing Conjecture \ref{MainProblem} is that the groupoid $\cG$ is in general non-compact,
and there is no obvious transversal structure to enable one to compare different $\bs$-fibers.

\subsection{Motivations and known results}
The most notable partial result for Conjecture \ref{MainProblem} is probably that of \cite{Heitsch;FoliHeat},
where the author proves that the heat kernel on the holonomy groupoid is smooth.
His proof uses heavily the constant rank of the (in general singular) foliation,
which in turn enables one to naturally lift the orthogonal complement of the foliation on $\rM$ to $\cG$. 

Not much is known about groupoids with non-constant rank.
The case of manifolds with boundary was proven by Melrose \cite{Melrose;Book},
using explicit construction with boundary defining functions.
Essentially the same argument was used by Albin \cite{Albin;EdgeInd}, generalizing the result to edge calculus.
However these techniques do not generalize easily to other groupoids, 
because there is no obvious definition of boundary function.

For groupoids not of the edge type, 
the only example we know is that of \cite{So;PhD}.
There, the author gives a criterion for the derivatives of the groupoid multiplication,
and then verify that the symplectic groupoid of the Bruhat Poisson structure satisfies these properties.
Our argument here will essentially follow similar lines.

Our main motivation of establishing the regularity is in line of the authors quoted above.
For example, the renormalized integral considered in 
\cite{Albin;EdgeInd} and \cite{So;PhD} requires the integrand to be differentiable.
That in turn means it is necessary to establish Conjecture \ref{MainProblem}
in order to obtain a renormalized index formula by the heat kernel method.

Another potentially interesting application of the results in this notes is on the polyhedral domain technique 
developed by Nistor et. al. 
The arguments developed in here may result in novel a-prior estimates for their numerical calculation.

Last but not the least,
this notes can been seen as a sequel of \cite{So;FullCal},
where the author attempts to generalize the classical theory of singular pseudo-differential calculus,
using groupoids only.
It appears the exponential coordinates defined here is central to many results 
(for example, smooth extension property) for boundary groupoids. 

\subsection{An overview of the argument} 
As far as we know, 
the exponential coordinates \cite{Nistor;IntAlg'oid} is the only way to construct local complements of the 
$\bs$-foliation.
In general, exponential coordinates are very arbitrary and therefore its use in estimating derivatives is limited.
In the case of boundary groupoid with exponential isotropy subgroups 
(which implies these Lie groups are solvable and simply connected),
the situation is easier as one only needs a finite cover of the singular invariant sub-manifolds.

In Section 2, we shall give a technical introduction of these exponential coordinates.
Then we study the change in coordinates formulas, 
as well as groupoid multiplication under these coordinates.

In Section 3, we first recall the construction of the heat kernel by the method of Levi parametix.
Then we directly show that on each exponential coordinate patch, 
derivatives of the Levi paramterix converges uniformly, thanks to the estimation done in Section 2. 
Hence we conclude that the groupoid heat kernel is smooth.

\section{Asymptotically Euclidean boundary groupoids}
\begin{dfn} 
Let $\cG \rightrightarrows \rM $ be a Lie groupoid with $\rM$ compact.
We say that $\cG$ is a {\it boundary groupoid} if 
\begin{enumerate}
\item
The anchor map $\nu : \rA \to T \rM $ stratifies $\rM $ into invariant sub-manifolds 
$\rM _0 , \rM _1 , \cdots , \rM _r \subset \rM $;
\item
For all $k = 0, 1, \cdots r$, 
$\bar \rM _k := \rM _ k \bigcup \rM _{k + 1} \bigcup \cdots \bigcup \rM _r $ are closed, 
immersed sub-manifolds of $\rM$;
\item
$\cG _{0 } := \bs ^{-1} (\rM _0 ) = \bt ^{-1} (\rM _0 ) \cong \rM _0 \times \rM _0$, the pair groupoid, and 
$\cG _{k } := \bs ^{-1} (\rM _k ) = \bt ^{-1} (\rM _k )
\cong \rG _k \times ( \rM _k \times \rM _k ) $ for some Lie groups $\rG _k $;
\item
For each $k$, there exists (unique) sub-bundles $\bar \rA _k \subset \rA |_{\bar \rM _k}$
such that $\bar \rA |_{\rM _k } = \ker (\nu |_{\rM _k }).$  
\end{enumerate}
\end{dfn}
For simplicity, we shall also assume that $\rG _k $ and $\rM _k $ are connected, hence all $\bs$-fibers are connected.

On any boundary groupoids we have the following fundamental estimates \cite{So;FullCal}.
Fix any Riemannian metric $\bar g $ on $\rM $.
For each $k \geq 1 $, let $d ( \cdot , \bar \rM _{k }) $ be the distance function defined by $ g $.
For each $k \geq 0$, fix a function $\rho _k \in C^\infty (\rM)$ such that 
$\rho _k > 0 $ on $\rM \setminus \bar \rM _k $ 
and $\rho _k = d ( \cdot , \bar \rM _k) $ on some open set containing $ \bar M _k $.
\begin{lem} 
\label{LocalDegen}
For each $k$, there exists a constant $\omega _k$ 
such that for any $ x $ lying in some open neighborhood of $\bar \rM _k$, $ X \in A _x ,$
\begin{equation}
\label{LocalDegEq}
| d \rho _k \circ \nu (X) | \leq \omega _k \rho _k (x) | X |_{g _\rA} .
\end{equation}
\end{lem} 

\begin{lem} 
\label{DEst}
For each $k$, let $\omega _k$ be defined in the previous Lemma \ref{LocalDegen}. 
Suppose further that $ | d \rho _k \circ \nu (X) | \leq \omega _k \rho _k (x) | X |_{g _\rA} $ for any $X \in \rA $.
Then for any $x \in \rM , a, b \in \cG _x $, 
\begin{equation}
\omega _k d (a, b ) \geq \left| \log \left( \frac{\rho _k (\bt (b))}{\rho _k (\bt (a))} \right) \right| .
\end{equation}
\end{lem}

In this notes, we consider the simple case 
$\cG = (\rM _0 \times \rM _0 ) \bigsqcup (\rM _1 \times \rM _1 \times \bbR ^q ).$
In the following we write $\omega := \omega _1 , \rho := \rho _1 $, 
and we let $\dim \rM _1 = p , \dim \rM _0 = n = p + q$.

In the same vein, we recall the definitions of uniformly degenerate and non-degenerate in \cite{So;FullCal}.
\begin{dfn}
The groupoid $\cG$ is said to be non-degenerate if there exist constants 
$\omega'_1 , \omega ' _2 , \cdots , \omega ' _r > 0 $ such that 
$$ |\nu (X) | \geq \omega ' _k \rho _k (x) | X |_{g _\rA} ,$$
for any $ x \in \bar \rM _{k - 1}, X \in A _x $, and $ X \bot \bar \rA _k ;$
The groupoid $\cG$ is said to be uniformly degenerate if there exist constants 
$\omega _1 , \omega_2, \cdots , \omega _r , \omega '_1  , \cdots \omega '_r > 0 $
and exponents $\lambda _1 , \cdots \lambda _r , \\ \lambda' _1 , \cdots , \lambda ' _r \geq 2$ such that 
$$ | d \rho _k \circ \nu (X) | \leq \omega _k (\rho _k (x))^{\lambda _k} | X |_{g _\rA} 
\text { and } |\nu (X) | \geq \omega ' _k (\rho _k (x))^{\lambda ' _k } | X |_{g _\rA} ,$$
for any $ x \in \bar \rM _{k-1}, X \in A _x $, and $ X \bot \bar \rA _k .$
\end{dfn}

\subsection{The exponential map}
Let $\cG \rightrightarrows \rM$ be a Lie groupoid with $\rM $ compact.
Denote the Lie algebroid of $\cG$ by $\rA \to \rM $ and anchor map by $\nu$.
Fix a Riemannian metric on $\rA$. 

Given any smooth section $X \in \Gamma ^\infty (\rA)$,
denote by $X ^\br$ the right invariant vector field on $\cG $ with $\bs ^* X ^\br = 0 $ and $X ^\br |_\rM = X $.
Since $\rM $ is compact, 
it is standard that $X　^\br$ is a complete vector field on $\cG$,
hence one has a well defined map 
$$ \exp X : \rM \to \cG ,$$
given by the flow of $X ^\br $ form each $x \in \rM \subset \cG $.
It is a well known fact that $\bt \circ \exp X  $ equals the flow of 
$ \nu (X) $ on $\rM $ and hence is a $\exp X $ is an admissible section.
Define
$$ E _X := d \bt \circ d ( \exp X |_ \rA ) : \rA \to \rA .$$

We list some basic properties of the exponential map \cite{Nistor;IntAlg'oid}, \cite{Mackenzie;Book}:
\begin{enumerate}
\item
For any $X , Y \in C^\infty (\rA ), \exp X \exp Y = \exp Y \exp E _X $;
\item
For any $x \in \rM , ((\exp X )(x))^{-1} = \exp (- X) (E ^\nu _X (x)) $,
where $E ^\nu _X : \rM \to \rM $ is the flow of $\nu (X)$.
\end{enumerate}

\begin{nota}
For any collection of sections $Z _I = (Z_1 , \cdots Z _{|I|} ) \in \Gamma ^\infty (\rA )$, denote 
$$ \exp Z _I := \exp Z _{|I|} \exp Z _{|I| - 1} \cdots \exp Z _2 \exp Z _1 .$$
For any $\mu = (\mu _1 , \cdots , \mu _{|I|}) \in \bbR ^{|I|} $, denote
$$ \exp ( \mu \cdot Z _I )
:= \exp \mu _{|I|} Z _{|I|} \exp \mu _{|I| - 1} Z _{|I| - 1} \cdots \exp \mu _2 Z _2 \exp \mu _1 Z _1 .$$
\end{nota}

We specialize the construction of exponential coordinates in \cite{Nistor;IntAlg'oid} to our case.
Fix an orthonormal bases $\{ Y _1 , \cdots , Y _q \} $ of $\bbR ^q $.
We regard it as a basis of 
$ T \rM _1 \times \bbR ^q $ and extend to an orthonormal set of sections on some neighborhood of $\rM _1 $,
which we still denote the extension by $\{ Y _1 , \cdots , Y _q \} $.
Let $\tilde V _\alpha \subset \rM \subset \cG$ be a coordinate patch on $\rM$, 
$\tilde V _\alpha \bigcap \rM _1 \neq \emptyset$.
We may assume that $\rA |_{\tilde V _\alpha } $ is trivial.
Fix a orthonormal basis
$$\{ Y _1 , \cdots Y _p , X ^{(\alpha)} _1 , \\ \cdots , X ^{(\alpha )} _q \} 
\subset \Gamma ^\infty (\rA |_{\tilde V _\alpha }),$$
and extend $X ^{(\alpha )}$ to smooth compactly supported sections. 

\begin{nota}
For any $\tau \in \bbR ^q $, denote by 
$$\| \tau \| := | \tau _1 | + \cdots + | \tau _q |, $$ 
the 1-norm of $\tau $.
Given $r , M > 0$, define 
$$ T (r , M) := \{ (x, \tau ) \in \rM \times \bbR ^q : \rho (x) < r e ^{- M \| \tau \| } \} .$$
\end{nota}
Clearly $T (r' , M') \subseteq T (r , M) $ whenever $r' \geq r, M ' \geq M $.

\begin{dfn}
Given $\{ Y _1 , \cdots Y _p , X ^{(\alpha)} _1 , \cdots , X ^{(\alpha )} _q \}$ as above,
and a arbitrary collection 
$Z _I = \{ Z _1 , \cdots , Z _{|I|} \} \subset \Span _\bbR \{ X ^{(\alpha)} _1 , \cdots , X ^{(\alpha )} _q \}, $
(over all $\alpha $) define the function $\bx ^{(\alpha )} _{Z _I} : V _\alpha \times \bbR ^{p + q} \to \cG $,
$$ \bx ^{(\alpha )} _{Z _I} (x , \mu , \tau ) 
:= \exp \tau \cdot Y \exp \mu \cdot X ^{(\alpha )} \exp {Z _I} (x) .$$
\end{dfn}

\begin{lem} 
\label{DiffeoLem}
There exists $\varepsilon , r _0 , M > 0$ such that $\bx ^{(\alpha )}$ 
is a local diffeomorphism from $(- \varepsilon , \varepsilon ) ^p \times T ( r _0 , M )$ onto its image.
\end{lem}
\begin{proof}
By the inverse function theorem it suffices to prove that the differential of $\bx $ is invertible.
It is straightforward to compute for any $ a = \bx (x , \mu , \tau) $,
the `coordinate vector fields' are given by
\begin{align}
\label{CoordVF}
\partial _{\tau _i} (a) &= 
d \br _a ( E _{\tau _q Y _q } \cdots E _{\tau _{i+1} Y _{i+1}}
Y _i (E ^\nu _{\tau _i Y _i} \cdots E ^\nu _{\tau _1 Y _1 } E ^\nu _{\mu \cdot X ^{(\alpha )}} E ^\nu _{Z _I} (x))), 
\\ \nonumber 
\partial _{\mu _{i'}} (a) &= 
d \br _a (E _{\tau \cdot Y } E _{\mu _p X ^{(\alpha )} _p } \cdots E _{\mu _{i'+1} X ^{(\alpha )} _{i'+1}}
Y _{i'} (E ^\nu _{\mu _{i'} X ^{(\alpha )} _{i'}} \cdots E ^\nu _{\mu _1 X ^{(\alpha )} _1 } E ^\nu _{Z _I} (x))).
\end{align}
In view that $d \bs (\partial _x ) = \partial _x $, it suffices to show that 
$ \partial _{\tau _i} , \partial _{\mu _{i'}} $ is a basis of $T ^\bot \cG $.
Following Equation (\ref{CoordVF}), one writes
\begin{align*}
E _{\tau _q Y _q }\cdots E _{\tau _{i +1} Y _{i +1}} Y _i 
(E ^\nu _{\tau _i Y _i} & \cdots E ^\nu _{\tau _1 Y _1 } 
E ^\nu _{\mu \cdot X ^{(\alpha )}} E ^\nu _{Z _I ^{\alpha }} (x)) \\
:=& \Big( \sum _{j =1} ^q w _{j i} Y _j + \sum _{j' = 1} ^p w _{j' i} X ^{(\alpha )} _{j'} \Big)
(E ^\nu _{\tau \cdot Y } E ^\nu _{ \mu \cdot X ^{(\alpha )}} E ^\nu _{Z _I ^{\alpha }}) (x)) \\
E _{\tau _\cdot Y} E _{\mu _p X ^{(\alpha )} _p} \cdots E _{\mu _{i' +1} X ^{(\alpha )} _{i' +1 }}& X ^{(\alpha )} _{i'}
(E ^\nu _{\mu _{i'} \cdot X ^{(\alpha )}_{i'}} \cdots E^\nu _{\mu _1 X ^{(\alpha )} _1} E ^\nu _{Z _I ^{\alpha }}(x)) \\
:=& \Big( \sum _{j =1} ^q w _{j i'} Y _j + \sum _{j' = 1} ^p w _{j' i'} X ^{(\alpha )} _{j'} \Big)
(E ^\nu _{\tau \cdot Y } E ^\nu _{ \mu \cdot X ^{(\alpha )}} E ^\nu _{Z _I ^{\alpha }}) (x)).
\end{align*}
Observe that at $x \in \rM _1 $, 
$ w _{j i} = 1 $ if $i = j$, $ w _{j' i'} = 1 $, and zero otherwise.
Moreover, it is clear that the $x$-derivatives of $w$ are bounded by $C _1 e ^{M _1 \| \tau \| }$.
It follows that for any $(x, \tau ) \in T (r , M _1 )$ with $r > 0$ sufficiently small,
we can write
\begin{equation}
\label{ExpPtb}
\left(
\begin{array}{cc}
w _{j i} & w _{j' i} \\
w _{j i'} & w _{j' i'}
\end{array}
\right)
= I + 
\left(
\begin{array}{cc}
w' _{j i} & w' _{j' i} \\
w' _{j i'} & w' _{j' i'}
\end{array} 
\right) ,
\end{equation}
with the last term on the right hand side uniformly small, hence 
$\left(
\begin{array}{cc}
w _{j i} & w _{j' i} \\
w _{j i'} & w _{j' i'}
\end{array} 
\right)$
is invertible, and the assertion follows.
\end{proof}

In the rest of this notes we shall make the following additional assumption:
\begin{hypo}
\label{Hypo}
For some $T ( r _0 , M )$,
the map $\bx ^{(\alpha )} $ are diffeomorphisms.
In other words, $\bx ^{(\alpha )} $ defines a local coordinate patch.
\end{hypo}

\begin{nota}
We denote the image of $\bx ^{(\alpha )} $ by $U _\alpha $.
For fixed $\mu , \tau $, denote $\bx ^{(\alpha )} (\mu , \tau ) $ to be the (local) section 
$\exp \tau \cdot Y \exp \mu \cdot X ^{(\alpha )} \exp Z _{I _\alpha } $.
\end{nota}

\begin{rem}
In general, if $\rG $ is not Abelian, then $| E _{\tau \cdot Y } | $ is not uniformly bounded.
If $\rG $ is nilpotent, then $| E _{\tau \cdot Y } | $ is bounded by estimate of the form $e ^{(\log |\tau |) ^2 } $.
See \cite[Appendix B]{So;FullCal}.
\end{rem}

\subsection{Groupoid multiplication}
The main advantage of the exponential coordinates we constructed, 
is that the groupoid multiplication can be described rather explicitly.

To begin with, consider $\exp \tau _i Y _i \bx ^{(\alpha )} (x , \mu , \tau )$.
Recall that by definition, 
$\exp t Y _i \\ \bx ^{(\alpha )} (x , \mu , \tau )$ is just the integral curve of the vector field 
$Y _i ^\br $ from $\bx ^{(\alpha )} (x , \mu , \tau )$, evaluated at $\tau _i $.
So we write
$$ Y _i ^\br (\bx ^{(\alpha )} (x, \mu , \tau ))
:= \sum _{j = 1} ^q v _{j i} (x, \mu , \tau ) \partial _{\tau _j } (x, \mu , \tau )
+ \sum _{j' = 1} ^p v _{j' i} (x, \mu , \tau ) \partial _{\mu _{j'}} (x, \mu , \tau ). $$
In other words,
\begin{align*}
Y _i (E ^\nu _{\tau ^{(\alpha )} \cdot Y } E ^\nu _{\mu ^{(\alpha )} \cdot X ^{(\alpha )}} E ^\nu _{X _I} (x))
=& \sum _{j = 1} ^q v _{j i} E _{\tau _q Y _q} \cdots E _{\tau _{j+1} Y _{j+1}} 
Y _j (E ^\nu _{\tau _j Y _j } \cdots E ^\nu _{\tau _1 Y _1} E ^\nu _{\mu \cdot X ^{(\alpha )}}(x)) \\ 
+ \sum _{j' = 1}^p v _{j' i} 
E _{\tau \cdot Y} & E _{\mu _p X ^{(\alpha )}_p} \cdots E _{\mu _{j'+1} X ^{(\alpha )}_{j'+1}} 
X ^{(\alpha )} _{j'} (E ^\nu _{\mu _{j'} X ^{(\alpha )}_{j'}} \cdots E ^\nu _{\mu ^{(\alpha )} _1 X ^{(\alpha )} _1}(x)).
\end{align*}
Note that 
$
\left(
\begin{array}{cc}
v _{j i} & v _{j' i} \\
v _{j i'} & v _{j' i'}
\end{array}
\right)
=
\left(
\begin{array}{cc}
w _{j i} & w _{j' i} \\
w _{j i'} & w _{j' i'}
\end{array}
\right) ^{-1} ,
$
with the right hand side defined in the proof of Lemma \ref{DiffeoLem}.

It is now straightforward to write
$$ \exp t Y _i \bx ^{(\alpha )} (x , \mu , \tau ) 
:= \bx ^{(\alpha )} (x , \varphi (x, \mu , \tau , t) , \psi (x, \mu , \tau , t)),$$
where 
$(\varphi (x, \mu , \tau , t), \psi (x, \mu , \tau , t))$ 
is just the solution of the system of ODE 
\begin{align}
\label{MultiODE}
\frac{d \varphi _{j'} (x, \mu , \tau , t)}{d t} 
=& v _{j' i} (x, \varphi (x, \mu , \tau , t), \psi (x, \mu , \tau , t)) \\ \nonumber
\frac{d \psi _j (x, \mu , \tau , t)}{d t} 
=& v _{j i} (x, \varphi (x, \mu , \tau , t), \psi (x, \mu , \tau , t)) 
\end{align}
with the obvious initial condition $\varphi (x, \mu , \tau , 0) = \mu , \psi (x, \mu , \tau , 0) = \tau $.

As in Equation (\ref{ExpPtb}), 
observe that when $x \in \rM _1 , \partial _{\tau _i } = Y _i ^\br$. 
In other words $v _{j i} = \delta _{j i}$, and $v _{j' i } = 0 $.
Clearly, the derivatives 
$ \partial _x v _{j i} (x , \mu , \tau ), \partial _x v _{j' i} (x ,\mu , \tau )$ 
are bounded by $C _2 e ^{ M _2 \| \tau \| }$.
It follows that we have the a-prior estimate 
$$| \psi _j (x , \mu , \tau , t ) - \tau _i | \leq \rho (x) C _2 e ^{M _2 \| \tau \| } t \leq C _2 r t,$$
whenever $x \in T (r, K), K \geq M _2 $.
In particular, if $(x, \tau ) \in T (r e ^{- q C _2 r t _0 } , K ) $,
then $\psi (x, \mu , \tau , t ) $ is well defined for $- t _0 \leq t \leq t _0 $ 
since $(x , \psi (x, \mu , \tau , t)) \in T (r , K ) $.

In turn, the a-prior estimate implies
\begin{align*}
\big| v _{j i} (x, \varphi (x, \mu , \tau , t), \psi (x, \mu , \tau , t)) - \delta _{j i} \big| 
\leq & \rho (x) C _2 e ^{ M _2 (\| \tau \| + q C _2 r t) } \\
\big| v _{j' i} (x, \varphi (x, \mu , \tau , t), \psi (x, \mu , \tau , t)) \big| 
\leq & \rho (x) C _2 e ^{ M _2 (\| \tau \| + q C _2 r t) }.
\end{align*}
Integrating with respect to $t$, we get
\begin{align} 
\label{CoordEst}
| \psi _j (x, \mu , \tau , t)  - \delta _{j i} t - \tau _j | 
\leq & \rho (x) (q r ) ^{-1} e ^{ M _2 (\| \tau \| + q C _2 r t) }  \\ \nonumber
| \varphi _{j'} (x, \mu , \tau , t) - \mu _{j'}| 
\leq & \rho (x) (q r ) ^{-1} e ^{ M _2 (\| \tau \| + q C _2 r t) } .
\end{align}

The case for $ \exp t Z \bx ^{( \alpha )} (x , \mu, \tau ) , t \in [0, 1]$, 
where $Z \in \Span \{ X ^{(\alpha )} \} $ (over all $\alpha $) is similar.
Again one writes 
$$ \exp t Z \bx ^{(\alpha )} (x , \mu , \tau ) 
= \bx ^{(\alpha )} (x , \tilde \varphi (x, \mu , \tau , t) , \tilde \psi (x, \mu , \tau , t)),$$
and one has
\begin{align}
\frac{d \tilde \varphi _{j'} (x, \mu , \tau , t)}{d t} 
=& v _{j' i'} (x, \tilde \varphi (x, \mu , \tau , t), \tilde \psi (x, \mu , \tau , t)) \\ \nonumber
\frac{d \tilde \psi _j (x, \mu , \tau , t)}{d t} 
=& v _{j i'} (x, \tilde \varphi (x, \mu , \tau , t), \tilde \psi (x, \mu , \tau , t)).
\end{align}
Using the same arguments as above, 
we get $w _{j i'} (x, \tilde \varphi , \tilde \psi ) = 0 $ for any $x \in \rM _1 $ and 
\begin{equation}
| \tilde \psi _j (x, \mu , \tau , t) - \tau _j | \leq \rho (x) (q r ) ^{-1} e ^{ M _2 (\| \tau \| + q C _2 r t) }. 
\end{equation}

\begin{rem}
Observe that one can also write 
$$ \exp t Z \bx ^{( \alpha )} (x , \mu, \tau ) 
= \exp _{\tau \cdot Y } \exp _{E _{\tau \cdot Y } t Z }
\bx ^{(\alpha )} (x , \mu , 0 ) (x) .$$  
\end{rem}

The techniques leading to 
Equation (\ref{CoordEst}) can be further refined to give an estimate of the derivatives of the multiplication map.
First consider $d (\exp t Y _i ) \partial _x (\bx ^{(\alpha )} (x , \tau , \mu ))$.
Clearly, one has
$$ d (\exp t Y _i ) \partial _x (\bx ^{(\alpha )} (x , \tau , \mu )) 
= \Big( \partial _x + \sum _{j = 1 } ^q (\partial _x \psi _j ) \partial _{\tau _j} 
+ \sum _{j' = 1 } ^p (\partial _x \varphi _{j'} ) \partial _{\mu _j} \Big) 
(\bx ^{(\alpha )} (x , \varphi , \psi)).$$
To estimate $\partial _x \psi (x, \mu, \tau , t) $ and $ \partial _x \varphi (x , \mu , \tau , t)$,
differentiate Equation (\ref{MultiODE}) with respect to $x$ to get
\begin{align}
\frac{d}{d t} \partial _x \varphi _{j'} & (x, \mu , \tau , t ) \\ \nonumber
=& \Big(\frac {\partial v _{j' i}}{\partial x} 
+ \sum _{i' = k} ^p \frac{ \partial \varphi _{k'} }{\partial x } \frac{ \partial v _{j' i}}{\partial \mu _{k'}}
+ \sum _{i = k} ^q \frac{ \partial \psi _k }{\partial x } \frac{ \partial v _{j' i}}{\partial \tau _k} \Big)
(x, \varphi (x, \mu , \tau , t ), \psi (x, \mu , \tau , t )) \\ \nonumber
\frac{d}{d t} \partial _x \psi _j & (x, \mu , \tau , t ) \\ \nonumber
=& \Big(\frac {\partial v _{j i}}{\partial x} 
+ \sum _{k' = 1} ^p \frac{ \partial \varphi _{k'} }{\partial x } \frac{ \partial v _{j i}}{\partial \mu _{k'}}
+ \sum _{k = 1} ^q \frac{ \partial \psi _k }{\partial x } \frac{ \partial v _{j i}}{\partial \tau _k} \Big)
(x, \varphi (x, \mu , \tau , t ), \psi (x, \mu , \tau , t )).
\end{align}
Moreover, $\partial _x \varphi (x, \mu , \tau , 0 ) = \partial _x \psi (x, \mu , \tau , 0 ) = 0 $.

Recall that 
$ |\partial _x v _{j _1} (x, \varphi (x, \mu , \tau , t), \psi (x, \mu ,\tau , t))| 
\leq \rho (x) (q r ) ^{-1} e ^{ M _2 (\| \tau \| + q C _2 r |t|) },$
for some constants $C _2 , M _2 > 0$.
As for the terms in the summation, 
note that at $x \in \rM _1 $, $v _{j 1} $ are just constants, 
therefore $\partial _{\tau _i} v _{j 1}, \partial _{\mu _{i'}} v _{j 1}$ vanish.
It follows that $\partial _{\tau _i} v _{j 1}, \partial _{\mu _{i'}} v _{j 1}$ are bounded for 
$(x, \tau) $ lying in some $ T (r _0 , K )$.
In other words, we have a linear equation of the form
$$
\frac{d}{d t} 
\left(
\begin{array}{c}
\partial _x \psi \\
\partial _x \varphi
\end{array}
\right)
(x, \mu , \tau , t)  
\leq A 
\left(
\begin{array}{c}
\partial _x \psi \\
\partial _x \varphi
\end{array}
\right)
(x, \mu , \tau , t)
+ \rho (x) (q r ) ^{-1} e ^{ M _2 (\| \tau \| + q C _2 r t) }.
$$
for some constant matrix $A$.
Integrating, one obtains: 
\begin{lem}
\label{DiffEstLem}
One has 
$$(d \exp t Y _i ) \partial _x (\bx ^{(\alpha )} (x , \tau , \mu )) = \partial _x + X ,$$
where $X \in T ^ \bot \cG $ and satisfies an estimate the form
$$ |X| \leq C _3 e ^{M _3 (\| \tau \| + |t|)} ,$$
for some constants $C _3, M _3 $. 
\end{lem}

The case for $\partial _{\tau _k } $ and $\partial _{\mu _{k'}}$ is much simpler.
Observe that both vector lies in $T ^\bot \cG $,
therefore
$$(d \exp t Y _i ) \partial _{\tau _k } (\bx ^{(\alpha )} (x , \tau , \mu )) 
= d \br _{ \exp _{t Y _i } \bx ^{(\alpha )} (x , \tau , \mu )}
(E _{t Y _i } (d \br ^{-1} _{\bx ^{(\alpha )} (x , \tau , \mu )} \partial _{\tau _k})). $$
Since the metric $g $ on $T ^\bot \cG $is right invariant, 
we have 
\begin{equation}
\label{DiffEstVert}
| (d \exp t Y _i ) \partial _{\tau _k } | \leq C _4 e ^{ M _4 t } |\partial _{\tau _k }|,
\end{equation}
for some constants $C _4, M _4 $,
and clearly similar estimate holds for $\partial _{\mu _{k'}}$.
Here, the important point to note is that the estimate does not depend on $(x, \mu, \tau )$.

\subsection{Change of coordinates}
\label{ChangeCoord}
In this section we turn to the change of coordinate formulas.
Observe that suppose one has
$\bx ^{(\beta )} (x, \mu , \tau ) 
= \bx ^{(\alpha )} (x, \varphi (x, \mu , \tau) , \psi (x, \mu , \tau )),$
then 
$$ \bx ^{(\beta )} (x, \mu , \tau ) = \exp \tau \cdot Y \exp \bx ^{(\beta )} (x, \mu , 0 )
= \exp \tau _\cdot Y \bx ^{(\alpha )} (x, \varphi (x, \mu , 0) , \psi (x, \mu , 0)) .$$
Suppose that $(\varphi (x, \mu , 0) - 2 \varepsilon , \varphi (x, \mu , 0) + 2 \varepsilon ) ^p 
\subset ( - \varepsilon _\alpha , \varepsilon _\alpha )^p $, where
$$\bx ^{(\alpha )} : ( - \varepsilon _\alpha , \varepsilon _\alpha )^p \times ( V _\alpha \bigcap T (r , K)) \to \cG .$$
Then, by possibly by restricting $T (r , K)$,
$\varphi (x, \mu , \tau) $ and $ \psi (x, \mu , \tau )$ can be computed using Equation (\ref{CoordEst}).
In particular, the second equation of (\ref{CoordEst}) implies 
$\varphi (x, \mu , \tau )$ still lies in the domain of $\bx ^{( \alpha )}$.
Moreover, repeated use of the first equation of (\ref{CoordEst}) now implies that 
\begin{align*}
\| \psi (x, \mu , \tau _1 , 0 , \cdots , 0 ) - \psi (x , \mu , 0, \cdots , 0 ) - \tau \| 
& \leq \rho (x) C _2 e ^{M _2 ( |\tau _1 | + \| \psi (x, \mu , 0 , \cdots , 0 ) \| )} \\
\| \psi (x, \mu , \tau _1 , \tau _2 , 0 , \cdots , 0 ) - \psi (x , \mu , \tau _1 , 0, \cdots , 0 ) - \tau \| 
& \leq \rho (x) C _2 e ^{M _2 ( |\tau _2 | + \| \psi (x, \mu , \tau _1 , 0 , \cdots , 0 ) \| )} \\
\leq \rho (x) C _2 & e ^{M _2 ( |\tau _2 | + | \tau _1 |
+ \rho (x) C _2 e ^{M _2 ( |\tau _1 | + \| \psi (x, \mu , 0 ) \| )}) } \\
& \vdots \\
\| \psi (x, \mu , \tau ) - \tau \| & \leq \rho (x) C _4 e ^{M _4 \| \tau \| }.
\end{align*}
Here, note the assumption $(x, \tau ) \in T (r, K) , K \geq C _2 $, 
implies all term on the right hand side are bounded.  
On the other hand, Lemma \ref{DiffEstLem} implies that one has estimations for the change of coordinate vector fields:
\begin{align*}
\|\partial _x \psi (x, \mu , \tau ) \| \leq & C _5 e ^{M _5 \| \tau \| }.
\end{align*}

Fix a collection of coordinate charts
$$\bx ^{(\alpha )} : ( - \varepsilon _\alpha , \varepsilon _\alpha )^p \times ( V _\alpha \bigcap T (r , C)) \to \cG ,$$
and $\varepsilon > 0 $
such that for all $x, \in V _\alpha , \mu ^{(\alpha )} \in ( - \varepsilon ^{(\alpha )} , \varepsilon ^{(\alpha )})$,
there exists $\beta $ (possibly same as $\alpha $) such that
$$ \bx ^{(\alpha )} (x, \mu ^{(\alpha )} , 0 ) = \bx ^{(\beta )} (x, \mu ^{(\beta )} , \tau ^{(\beta )}) $$
for some $ \mu ^{(\beta )} \in \bbR ^p $ satisfying
$ (- 2 \varepsilon + \mu ^{( \beta )}, 2 \varepsilon + \mu ^{( \beta )} ) \subseteq 
(- \varepsilon _ \beta , \varepsilon _\beta )^p .$

Since the product
\begin{equation}
\label{SmallChain}
\bx ^{(\alpha _2)} (E ^\nu _{\bx ^{(\alpha _1)} (\mu _1, \tau _1 )} (x) , \mu _2, \tau _2 )
\bx ^{( \alpha _1 )} (x, \mu _1 , \tau _1 ) 
\end{equation}
is by definition just an iteration of left multiplication by admissible sections of the form
$\exp Z $ then followed by $\exp \tau _2 \cdot Y $.
In particular, we may change coordinates so that
$$ \bx ^{( \alpha _1 )} (x, \mu _1 , \tau _1 ) = \bx ^{( \beta _1 )} (x, \mu ' _1 , \tau '_1 ), $$
where $ (- \varepsilon + \mu ', \varepsilon + \mu ' ) \subseteq (- \varepsilon _\beta , \varepsilon _\beta )^p .$
Then by Equation (\ref{CoordEst}) we have 
$$ \exp Z _1 \bx ^{( \beta _1 )} (x, \mu ' _1 , \tau '_1 )
= \bx ^{( \beta _1 )} (x, \mu '' _1 , \tau ''_1 ) ,$$
such that $ \| \tau ''_1 - \tau _1 \| \leq \rho (x) C _6 e ^{M _6 \| \tau _1 \| } $
for some $C _6 , M _6 \geq 0 $.
Obviously the same arguments can be iterated to compute (\ref{SmallChain}).

At the same time, 
one can estimate of the differential of each successive left multiplication by Lemma \ref{DiffEstLem} and 
Equation (\ref{DiffEstVert}).
Hence we conclude that
\begin{thm}
\label{ChainGrowthLem}
There exist some $r , H > 0$ such that whenever $x \in B (r e ^{- k H }), \| \tau _i \| \leq 1 , i = 1, 2, \cdots k$,
$$ \bx ^{(\alpha _k)} ( \mu _k , \tau _k ) \cdots \bx ^{(\alpha _1)} (\mu _1, \tau _1 ) (x) 
= \bx ^{(\beta )} (x, \mu , \tau ) $$
for some exponential coordinate chart $\bx ^{(\beta )}$.
Moreover, one has
$$d (\bx ^{(\alpha _k)} ( \mu _k , \tau _k ) \cdots \bx ^{(\alpha _1)} (\mu _1, \tau _1 )) (\partial _x )
= \partial _x + V ,$$
where $\partial _x $ on the right hand side is the coordinate vector field on $\bx ^{(\beta )}$,
and $V \in T ^\bot \cG $ (i.e. $V$ is spanned by 
$\partial _{\tau ^{( \beta )} _j } $ and $\partial _{\mu ^{(\beta )} _{j'}} $),
is such that 
$$| V | \leq C _7 e ^{M _7 k } ,$$ 
for some constant $C _7 , M _7 > 0$.
\end{thm}

\section{The groupoid heat kernel}

In terms of reduced kernel and convolution product, 
the heat kernel is defined as:
\begin{dfn}
A {\it (groupoid) Heat kernel} of $\Delta $ is a continuous section 
$$ Q \in \Gamma ^0 (\cG \times (0 , \infty ) ) ,$$
such that $ Q (a , t) , Q (a ^{-1} , t) $ are smooth when restricted to all $\cG _x \times (0 , \infty ) $,
and satisfies: 
\begin{enumerate}
\item
The heat equation
$$ (\partial _t + \Delta ) Q (a, t) = 0; $$
\item
The initial condition
$$ \lim _{t \rightarrow 0^+} Q \circ u = u, 
\quad \forall u \in \Gamma ^\infty _c (\cG), $$
where $\circ $ denotes the convolution product
$$ \kappa _0 \circ \kappa _1 (a)
:= \int _{ b \in \cG _(a)} \kappa _0 (a b ^{-1} ) \kappa _1 (b) d b , 
\quad \Forall \kappa _1 , \kappa _2 \in C ^0 _c (\cG ), a \in \cG .$$
\end{enumerate}
\end{dfn}

\subsection{Construction of the heat kernel}
In this section, 
we briefly review how one constructs the heat kernel of a generalized Laplacian operator \cite{BGV;Book}.
For simplicity, we only consider the scalar case.

By \cite{BGV;Book}, there exists a formal power series 
$ \Phi (a, t) := \sum _{i=1} ^\infty t^i \Phi _i (a) $
that satisfies the equation
$$ (\partial _t + \Delta ) (4 \pi t)^{- \frac{n}{2}} \: e^{- \frac{d (a , \bs (a))^2 }{4 t}} \Phi (a, t) = 0. $$
Fix a cutoff function $\phi $ supported on $B (\rM , \varrho _0)$ such that $\phi = 1 $ on the smaller set
$B (\rM , \frac{\varrho _0}{2}) := \{ a \in \cG : d (a , \bs (a)) \leq \frac{\varrho _0}{2} \}.$
Write 
$$ G _N (a, t) := \phi (a) (4 \pi t)^{- \frac{n}{2}} \: e^{- \frac{d (a , \bs (a))^2 }{4 t}}
\sum _{i=1}^N t^i \Phi _i (a), \quad t \in (0, \infty ).$$
For each $N > \frac{n}{2}$, define the sections 
$ R ^{(k)} _n \in \Gamma ^\infty (\cG \times [0, \infty ))$:
\begin{align*}
R ^{(1)} _N :=& \: (\partial _t + \Delta ) G _N \\
R ^{(k)} _N :=& \: \int _0 ^t R _N (\cdot , t - s ) \circ R ^{(k-1)} _N (\cdot , s ) d s \\
=& \: \int _0 ^t \int _{\bs ^{-1} (a)} R_n (a b^{-1} , t - s ) R^{(k-1)} _N (b , \tau ) \mu _{\bs (a)} (b) d s  \\
Q ^{(0)} _N :=& \: G _N \\ 
Q ^{(k)} _N :=& \: \int _0 ^t G _N (\cdot , t - s ) \circ R ^{(k)} _N (\cdot , s ) d s, \quad k \geq 1 \\
=& \: \int _{s \in \Sigma _k } G _N (\cdot , t - \| s \| ) \circ R ^{(1)} _N (\cdot , s _1 ) \circ 
\cdots \circ R ^{(1)} _N (\cdot , s _k ) d s,
\end{align*}
where 
$\Sigma _k := \{ (s _1 , \cdots , s _k ) \in \bbR ^k , s _1, \cdots , s _k \geq 0 ,
s _1 + \cdots + s _k \leq t \} ,$
and $d s $ is the Lebesgue measure.
Then the heat kernel is constructed by taking the sum
$ Q := \sum _{k=0} ^\infty (-1)^k Q ^{(k)} _N (\cdot , t). $

To simplify notation, we shall write $R _N := R ^{(1)} $, and omit the reference to $t $.

\subsection{$C ^l $ estimation near $\cG _1 $}
Let $\{ \bx ^{(\alpha )} \}$ be the family coordinate patches defined in Section \ref{ChangeCoord}.
Define 
$ U ^0 _\alpha := \{ \bx ^{(\alpha )} (x , \mu , \tau ) : \tau \in (-1 , 1)^q \} $
and $U ^0 := \bigcup _{\alpha } U ^0 _\alpha .$
Fix a partition of unity $\theta _\alpha $ of $U$ subordinated to $U _\alpha $.

In the construction of $G _N$, 
we may assume that $\varrho $ is sufficiently small,
such that 
$$ (\bs ^* \rho ) ^{-1} [0, \frac{r}{2})  \bigcap B (\rM , \varrho ) \subset U ^ 0 .$$
Fix another cutoff function $\chi \in C ^\infty _c (\bbR )$ such that $\chi = 1 $ on $[0, \frac{1}{2}]$
and $0$ on $[1, \infty )$.

Given smooth compactly supported functions $\kappa _0 , \cdots , \kappa _k$ on $\cG$,
it is straightforward to write for any $a \in \cG$:
\begin{align*} 
\kappa _0 \circ \cdots \circ & \kappa _k (a) \\
=& \int _{b _1 \in \cG _{\bs (a)}} \int _{ b _2 \in \cG _{\bt (b _1)}}
\cdots \int _{ b _k \in \cG _{\bt (b _{k -1})}} \kappa _0 (a b _1 ^{-1} \cdots b _k ^{-1} )
\kappa _1 (b _1 ) \cdots \kappa _k (b _k) d b_k \cdots d b _1 .
\end{align*}
Suppose that $\kappa _i $ is supported on $U ^0 $ for all $i$. 
Then one can write
\begin{align*}
(\kappa _0 \circ \cdots \kappa _k )(a)
&= \sum _{\alpha _1 , \cdots , \alpha _k } \int _{(\mu _i , \tau _i) \in (-1 , 1) ^{k (p + q)}}
\kappa _0 \big( a ( \bx ^{(\alpha _k)} (\mu _k , \tau _k ) \cdots \bx ^{(\alpha _1)} (\mu _1, \tau _1 )(x))^{-1} \big)\\
\times & \prod _{i = 1} ^k \theta _{\alpha _i} \kappa _i 
\big(\bx ^{(\alpha _i)} (E ^\nu _{\bx ^{(\alpha _{i - 1})} (\mu _{i - 1}, \tau _{i - 1} )} 
\cdots E ^\nu _{\bx ^{(\alpha _1)} (\mu _1, \tau _1 )} (x) , \mu _i , \tau _i ) \big) 
\prod _{i = 1} ^k d \mu _i d \tau _i.
\end{align*}
Consider differentiating the integrand with respect to $x$.

It is clear that 
$( d E ^\nu _{\bx ^{(\alpha _{i - 1})} (\mu _{i - 1}, \tau _{i - 1} )} 
\cdots E ^\nu _{\bx ^{(\alpha _1)} (\mu _1, \tau _1 )} ) (\partial _x ) \leq C _8 e ^{M _8 i }$,
for some constant $C _8 , M _8 > 0 $.
It follows that
$$ \big| L _{\partial _x } (\theta _{\alpha _i} \kappa _i (\bx ^{(\alpha _i)} 
( E ^\nu _{\bx ^{(\alpha _{i - 1})} (\mu _{i - 1}, \tau _{i - 1} )} 
\cdots E ^\nu _{\bx ^{(\alpha _1)} (\mu _1, \tau _1 )} (x) , \mu _i , \tau _i ) ) \big|
\leq |d \theta _{\alpha _i} \kappa _i | | \partial _x | C _8 e ^{M _8 i }.$$

We turn to the derivatives of
$\hat \kappa (a) 
:= \kappa _0 (a ( \bx ^{(\alpha _k)} ( \mu _k , \tau _k ) \cdots \bx ^{(\alpha _1)} (\mu _1, \tau _1 ) (x) ) ^{-1})$
on some exponential coordinates patch $\bx ^{(\beta )}$.
One has
\begin{align*}
\hat \kappa (\bx ^{(\beta )} & (x, \mu, \tau )) \\
=& \kappa ( \bx ^{(\beta )} ( \mu, \tau )
\hat \bx ^{(\alpha _k)} ( \mu _k , \tau _k ) \cdots \hat \bx ^{(\alpha _1)} (\mu _1, \tau _1 ) 
(E ^\nu _{\bx ^{(\alpha _{i - 1})} (\mu _{i - 1}, \tau _{i - 1} )} 
\cdots E ^\nu _{\bx ^{(\alpha _1)} (\mu _1, \tau _1 )} (x)) ,
\end{align*}
where 
$\hat \bx ^{(\alpha _i)} ( \mu , \tau ) 
:= \exp _{Z _1} \cdots \exp _{Z _{|I|}} \exp _{\mu _1 X _1 ^{(\alpha _i )} } \cdots \exp _{\mu _p X _p ^{(\alpha _i )} }
\exp _{\tau _1 Y _1} \cdots \exp _{\tau _q Y _q }$.
Suppose further that $x \in B (r e ^{- k (H + (|I| + p + q) \omega ) }, \rM _1 )$, then 
$$E ^\nu _{\bx ^{(\alpha _{i - 1})} (\mu _{i - 1}, \tau _{i - 1} )} 
\cdots E ^\nu _{\bx ^{(\alpha _1)} (\mu _1, \tau _1 )} (x) \in B (r e^{- k H } , \rM _1 ) ,$$ 
and from Theorem \ref{ChainGrowthLem}, one can write:
$$ \hat \kappa (\bx ^{(\beta )} (x, \mu, \tau )) 
= \kappa \big( \bx ^{(\beta ')} (E ^\nu _{\bx ^{(\alpha _{i - 1})} (\mu _{i - 1}, \tau _{i - 1} )} 
\cdots E ^\nu _{\bx ^{(\alpha _1)} (\mu _1, \tau _1 )}) (x) , \varphi (x, \mu, \tau) , \psi (x, \mu , \tau ) \big).$$
Note that since $\kappa _0 $ is supported on $U ^0$,
the right hand side vanishes if $\| \tau \| > k H _1 $ for some $H _1$. 
Moreover, by the second part of Theorem \ref{ChainGrowthLem}, one has the estimate
$$ \big| L _{\partial _x} \hat \kappa (\bx ^{(\beta )} (x, \mu, \tau )) \big| 
\leq C _9 e ^{M _9 k } |d \kappa _0 |.$$
In particular, put
\begin{align*}
\kappa _0 (a) :=& \chi (2 r ^{-1} \bs ^* \rho (a)) G _N (a), \\
\kappa _i (a) :=& \chi (2 r ^{-1} \bs ^* \rho (a)) R _N (a) , \quad i \leq k - 1, \\
\kappa _k (a) :=& \chi (2 r ^{-1} H ^k \bs ^* \rho (a)) R _N (a).
\end{align*}
Differentiating under the integral sign and summing over all $\alpha _i $, we conclude that
\begin{lem}
\label{OnEst}
One has the uniform estimate
\begin{align}
\big|\partial _x \big( \chi (2 r ^{-1} \bs ^* \rho ) G _N \circ \chi (2 r ^{-1} & \bs ^* \rho ) R _N 
\circ \cdots \circ \chi (2 r ^{-1} \bs ^* \rho ) R _N \\ \nonumber
&\circ \chi (2 r ^{-1} H ^k \bs ^* \rho) R _N \big) (\bx ^{(\beta )} (x, \mu, \tau )) \big| \leq C _{10} e ^{M _{10} k },
\end{align}
for some $C _{10} , M _{10}$ (note that $M _8 $ is independent of $s , N$). 
Moreover the same estimate holds for 
the $\partial _{\tau _j } $ and $\partial _{\mu _{j'}} $ derivatives.
\end{lem}

\subsection{$C ^1 $ estimation away from $\cG _1 $}
In the following, we make the additional assumption that $\cG $ is non-degenerate or uniformly degenerate.
Denote by $\iota : \rM _0 \times \rM _0 \to \cG $ the embedding of the invariant sub-manifold. 
Since the target and source maps on 
$\rM _0 \times \rM _0$ are respectively the projection onto the first and second factor,
one has 
$$(\bt \times \bs ) \circ \iota = \id _{ \rM _0 \times \rM _0 }.$$
Using the fact that $d \bt ( \tilde V )= \nu (d \br  ^{-1} _a (\tilde V))$, 
for any $\tilde V \in T _a \cG , a \in \cG$ satisfying $ d \bs (\tilde V) = 0$,
it follows that 
\begin{equation}
\label{nuEst1}
d \iota ( V \oplus 0) = d \br _{\iota (x, y)} (\nu ^{-1} (V)),
\end{equation}
for any $V \oplus 0 \in T _{(x , y) } \rM _0 \times \rM _0 $.
Similarly,  
\begin{equation}
\label{nuEst2}
d \iota (0 \oplus W) = d (\bi \circ \br _{\iota (y, x)}) (\nu ^{-1} (W)).
\end{equation}
It is straightforward to compute the coordinate vector fields for the exponential coordinates:
\begin{align*}
d \bt (\partial _{\tau _i} (x, \mu , \tau)) 
=& \nu ( E _{\tau _q Y _q } \cdots E _{\tau _{i+1} Y _{i+1}} Y _i ) 
(E ^\nu _{\tau _i Y _i} \cdots E ^\nu _{\tau _1 Y _1} E ^\nu _{\mu \cdot X ^{(\alpha )}} 
E ^\nu _{Z _{I ^{(\alpha )}}} (x)) \\
d \bt (\partial _{\mu _{i'}} (x, \mu , \tau)) 
=& \nu ( E ^\nu _{\tau \cdot Y} E _{\mu _p X _p ^{(\alpha )} } \cdots E _{\mu _{i'+1} X _{i'+1} ^{(\alpha )}} 
X _{i'}^{(\alpha )}) 
(E ^\nu _{\mu _{i'} X ^{(\alpha )} _{i'}} \cdots E ^\nu _{\mu _1 X _1 ^{(\alpha )}} E ^\nu _{Z _{I ^{(\alpha )}}} (x)).
\end{align*}
Similarly, one computes
\begin{align*} 
d \bt (\partial _x (x, \mu , \tau )) =& \: d (E ^\nu _{\tau \cdot Y}
E ^\nu _{\mu \cdot X ^{(\alpha )}} E ^\nu _{Z _{I ^{(\alpha )} }}) (\partial _{x } (x)) \\
d \bs (\partial _x (x, \mu , \tau )) =& \: \partial _{x } (x),
\end{align*}
where $\partial _{x } (x)$ on the right hand side is regarded as a tangent vector on $\rM $.

On $\cG _0 = \rM _0 \times \rM _0 $, one can write 
\begin{align*}
(\kappa _0 \circ \cdots \circ \kappa _k ) (a) 
=& (\bt \times \bs ) ^* \int \iota ^* \kappa _0 (\bt (a) , b _1) 
\big( \prod _{l = 1} ^{k -1 } \iota ^* \kappa _l (b _l , b _{l+1}) \big)
\iota ^* \kappa _k (b _k , \bs (a)) d b _1 \cdots d b _k . 
\end{align*}
It is then straightforward to differentiate along the coordinate vector fields:
\begin{align*} 
\partial _{x } & (\kappa _0 \circ \cdots \circ \kappa _k ) (a) \\
=& 
\int \big( L _{ d E ^\nu _{\tau \cdot Y} E ^\nu _{\mu \cdot X ^{(\alpha )}} E ^\nu _{Z _{I ^{(\alpha )}}} 
\partial _x \oplus 0} \iota ^* \kappa _0 \big) (\bt (a) , b _1) 
\big( \prod _{l = 1} ^{k -1 } \iota ^* \kappa _l (b _l , b _{l+1}) \big)
\iota ^* \kappa _k (b _k , \bs (a)) d b _1 \cdots d b _k \\
&+ \int \iota ^* \kappa _0 (\bt (a) , b _1) 
\big( \prod _{l = 1} ^{k -1 } \iota ^* \kappa _l (b _l , b _{l+1}) \big) 
\big( L _{ 0 \oplus \partial _x } \iota ^* \kappa _k (b _k , \bs (a)) \big) d b _1 \cdots d b _k \\
\partial _{\tau _i} & (\kappa _0 \circ \cdots \circ \kappa _k ) (a) \\
=& \int \big( L _{ \nu ( E _{\tau _q Y _q } \dots  E _{\tau _{i+1} Y _{i+1}} Y _i ) \oplus 0 } 
\iota ^* \kappa _0 \big) (\bt (a) , b _1) 
\big( \prod _{l = 1} ^{k -1 } \iota ^* \kappa _l (b _l , b _{l+1}) \big)
\iota ^* \kappa (b _k , \bs (a)) d b _1 \cdots d b _k \\
\partial _{\mu _{i'}} & (\kappa _0 \circ \cdots \circ \kappa _k ) (a) \\
=& \int \big( L _{ \nu ( E ^\nu _{\tau \cdot Y} E _{\mu _p X _p ^{(\alpha )} } 
\cdots E _{\mu _{i'+1} X _{i'+1} ^{(\alpha )}} 
X _{i'}^{(\alpha )} ) \oplus 0 } 
\iota ^* \kappa _0 \big) (\bt (a) , b _1) \\
& \big( \prod _{l = 1} ^{k -1 } \iota ^* \kappa _l (b _l , b _{l+1}) \big)
\iota ^* \kappa (b _k , \bs (a)) d b _1 \cdots d b _k .
\end{align*}

Using Equation (\ref{nuEst2}) and the hypothesis that $\cG$ is non-degenerate or uniformly degenerate, 
we have the estimate
\begin{equation}
\label{Off1}
|L _{ 0 \oplus \partial _x } \iota ^* \kappa _k (b _k , \bs (a)) |
= |( L _{ \nu ^{-1} (\partial _x) ^\br}  \bi ^* \kappa _k) (\iota (\bs (a)) , b _k ) |
\leq |\partial _x||d \kappa _k | (\omega '\rho ( \bs (a) ) )^{- \lambda '},
\end{equation}
Using same arguments with $ 0 \oplus \partial _x$ replaced by 
$d E ^\nu _{\tau \cdot Y} E ^\nu _{\mu \cdot X ^{(\alpha )}} E ^\nu _{Z _{I ^{(\alpha )}}} \partial _x \oplus 0$, 
one gets 
\begin{equation}
\label{Off2}
\big| 
L _{ d E ^\nu _{\tau \cdot Y} E ^\nu _{\mu \cdot X ^{(\alpha )}} E ^\nu _{Z _{I ^{(\alpha )}}} \partial _x \oplus 0 } 
\iota ^* \kappa _0 (\bt (a) , b _1 ) \big|
\leq C _8 e ^{M _8 \| \tau \| } | \partial _x | |d \kappa _0 | (\omega ' \rho (\bs (a))) ^{- \lambda '}.
\end{equation}
(Here we used the fact that 
$ | d E ^\nu _{\tau \cdot Y} E ^\nu _{\mu \cdot X ^{(\alpha )}} E ^\nu _{Z _{I ^{(\alpha )}}} \partial _x |
\leq C _8 e ^{M _8 \| \tau \| } | \partial _x |$.)
Again, put 
\begin{align*}
\kappa _0 &:= (1 - \chi (2 r ^{-1} \bs ^* \rho )) G _N , \text { or } \chi (2 r ^{-1} \bs ^* \rho ) G _N , \\
\kappa _i &:= (1 - \chi (2 r ^{-1} \bs ^* \rho )) G _N , \text { or } \chi (2 r ^{-1} \bs ^* \rho ) R _N , 
\quad i \leq k - 1, \\
\kappa _k &:= (1 - \chi (2 r ^{-1} H ^k \bs ^* \rho )) R _N , \text { or } \chi (2 r ^{-1} H ^k \bs ^* \rho ) R _N ,
\end{align*}
but except the combination appeared in the last section. 
Observe that 
$$\iota ^* \kappa _0 (\bt (a) , z _1) 
\big( \prod _{l = 1} ^{k - 1 } \iota ^* \kappa _l (z _l , z _{l+1}) \big)
\iota ^* \kappa _k (z _k , \bs (a)) = 0,$$
unless $\rho (\bs (a) ) \geq H ^{- k} e ^{- k \omega \chi _0 }$. 
Together with Equations (\ref{Off1}) and (\ref{Off2}),
one gets the uniform bound
$$ \Big| L _{\partial _x} \big( \iota ^* \kappa _0 (\bt (\bx ^{(\alpha )} (x, \mu , \tau ) , b _1) 
\big( \prod _{l = 1} ^{k -1 } \iota ^* \kappa _l (b _l , b _{l+1}) \big)
\iota ^* \kappa _k (b _k , x) \big) \Big|
\leq C _{11} e ^{M _{11} k} .$$
In other words, one has
\begin{align}
\label{OffEst} \nonumber
\big| \partial _x \big( - \chi (2 r ^{-1} \bs ^* \rho ) G _N \circ \chi (2 r ^{-1} \bs ^* \rho ) R _N 
\circ & \cdots \circ \chi (2 r ^{-1} \bs ^* \rho ) R _N \circ \chi (2 r ^{-1} H ^k \bs ^* \rho) R _N \\ 
+ G _N \circ R _N & \circ \cdots \circ R _N \big) (\bx ^{(\beta )} (x, \mu, \tau )) \big| \leq C _{12} e ^{M _{12} k } ,
\end{align}
and similar for other coordinate vector fields.

\subsection{Regularity of the heat kernel}
Adding Equation (\ref{OffEst}) and Lemma \ref{OnEst} together,
and using the fact that the volume of the set 
$\Sigma \subset \bbR ^k $ equals $(k !) ^{-1} t ^k $, it follows that the sum 
$$ \sum _{k=0} ^\infty (-1)^k Q ^{(k)} _N (\cdot , t). $$
converges uniformly and absolutely in $C ^1$ for $N$ sufficiently large.

Clearly, similar arguments holds for all $l = 1, 2, \cdots $
(in particular Equation (\ref{MultiODE}) can be differentiated repeatedly 
and estimates in the form of Lemma \ref{DiffEstLem} still holds).
Hence we conclude that
\begin{thm}
Let $\cG$ be a boundary groupoid of the form 
$\cG = \rM _0 \times \rM _0  \bigsqcup \rM _1 \times \rM _1 \times \bbR ^q $,
that is either non-degenerate or uniformly degenerate.
Suppose furthermore that there exists a finite collection of exponential coordinates charts 
$\{ \bx ^{(\alpha )} \} $ satisfying Assumption \ref{Hypo}.
Then the heat kernel $Q $ lies in $C ^\infty (\cG \times (0, \infty ))$.
\end{thm}

\section*{Acknowledgements}
The author would like to thank Victor Nistor for many useful discussions. 
This project is supported by the AFR(Luxembourg) postdoctoral fellowship.


\begin{thebibliography}{10}

\bibitem{Albin;EdgeInd}
P.~Albin.
\newblock A renormalized index theorem for some complete asymptotically regular
  metrics: the {Gauss-Bonnet} theorem.
\newblock {\em Adv. in Maths,}, {\bf 213}(1):1--52, 2007.

\bibitem{BGV;Book}
N.~Berline, E.~Getzler, and M.~Vergne.
\newblock {\em Heat kernels and Dirac operators}.
\newblock Springer-Verlag, 1992.

\bibitem{Fern'd;HoloAndChar}
R.L. Fernandes.
\newblock Lie algebroids, homonomy and characteristic classes.
\newblock {\em Adv. in Maths.}, {\bf 170}:119--179, 2002.

\bibitem{Heitsch;FoliHeat}
J.~L. Heitsch.
\newblock {Bismut} super-connections and the {Chern} characters for {Dirac}
  operators on foliated manifolds.
\newblock {\em K-Theory}, {\bf 9}:507--528, 1995.

\bibitem{Mackenzie;Book}
K.~Mackenzie.
\newblock {\em Lie groupoids and Lie algebroids in differential geometry}.
\newblock Cambridge University Press, 1987.

\bibitem{Melrose;Book}
R.B. Melrose.
\newblock {\em The Atiyah-Patodi-Singer index theorem}.
\newblock A K Peters, 1993.

\bibitem{Nistor;IntAlg'oid}
V.~Nistor.
\newblock Groupoids and integration of {Lie} algebroids.
\newblock {\em J. Math. Soc. Japan}, {\bf 52}(4):847--868, 2000.

\bibitem{NWX;GroupoidPdO}
V.~Nistor, A.~Weinstein, and P.~Xu.
\newblock Pseudodifferential operators on differential groupoids.
\newblock {\em Pac. J. Maths}, {\bf 189}(1):117--152, 1999.

\bibitem{So;PhD}
B.K. So.
\newblock {\em Pseudo-differential operators, heat calculus and index theory of
  groupoids satisfying the {Lauter-Nistor} condition}.
\newblock PhD thesis, The University of Warwick, 2010.

\bibitem{So;FullCal}
B.K. So.
\newblock On the full calculus of pseudo-differential operators on boundary
  groupoids with polynomial growth.
\newblock preprint, arXiv.org/abs/1111.7274, 2011.

\end{thebibliography}
\end{document}